\documentclass{amsart}
\usepackage[utf8]{inputenc}
\usepackage[T1]{fontenc}
\usepackage{amsmath}
\usepackage{amsfonts}
\usepackage{amssymb}
\usepackage{bbm}
\usepackage{tikz}
\usepackage{color}
\newcommand\CG[1]{\textcolor{red}{#1}}

\def\R{{\mathbb R}}
\def\C{{\mathbb C}}

\def\<{\langle}
\def\>{\rangle}

\def\0{\underline 0}

\def\1{\underline 1}

\newcommand{\bel}{\begin{equation}\label}
\newcommand{\ee}{\end{equation}}

\newtheorem{lemma}{Lemma}

\newtheorem{example}{Example}
\newtheorem{proposition}{Proposition}
\newtheorem{theorem}{Theorem}
\newtheorem{remark}{Remark}

\newtheorem{corollary}{Corollary}

\begin{document}
\title[Kummer and gamma laws through independences on trees ]{Kummer and gamma laws through independences on trees --- another parallel with the Matsumoto--Yor property}
\author{Agnieszka Piliszek, Jacek Weso\l owski}
\address{Agnieszka Piliszek\\
Wydzia{\l} Matematyki i Nauk Informacyjnych\\
Politechnika Warszawska\\
Koszykowa 75\\
00-662  Warszawa, Poland}
\email{A.Piliszek@mini.pw.edu.pl}
\address{Jacek Weso{\l}owski\\
Wydzia{\l} Matematyki i Nauk Informacyjnych\\
Politechnika Warszawska\\
Koszykowa 75\\
00-662  Warszawa, Poland}
\email{wesolo@mini.pw.edu.pl}
\begin{abstract}
The paper develops a rather unexpected parallel to the multivariate Matsumoto--Yor (MY) property on trees considered in \cite{MW04}. The parallel concerns a multivariate version of the Kummer distribution, which is generated by a tree. Given a tree of size $p$, we direct it by choosing a vertex, say $r$, as a root. With such a directed tree we associate a map $\Phi_r$. For a random vector ${\bf S}$ having a $p$-variate tree-Kummer distribution and any root $r$, we prove that $\Phi_r({\bf S})$ has independent components. Moreover, we show that if ${\bf S}$ is a random vector in $(0,\infty)^p$ and  for any leaf $r$ of the tree the components of $\Phi_r({\bf S})$ are independent, then one of these components has a Gamma distribution and the remaining $p-1$ components have Kummer distributions. Our point of departure is a relatively simple independence property due to \cite{HV15}. It states that if $X$ and $Y$ are independent random variables having Kummer and Gamma distributions (with suitably related parameters) and $T:(0,\infty)^2\to(0,\infty)^2$ is the involution defined by $T(x,y) =(y/(1+x), x+xy/(1+x))$, then the random vector $T(X,Y)$ has also independent components with Kummer and gamma distributions.
By a method inspired by a proof of a similar result for the MY property, we show that this independence property characterizes the gamma and Kummer laws.
\end{abstract}

\maketitle

\section{Introduction}
\label{sec:1}

Let $X$ and $Y$ be independent random variables. There are several well-known cases where $U=\phi(X,Y)$ and $V=\psi(X,Y)$ are also independent. A number of distributions have actually been characterized this way. Classical results along these lines include Bernstein's characterization \cite{Be41} of the Gaussian distribution through independence of $U=X-Y$ and $V=X+Y$, and Lukacs' characterization \cite{Lu55} of the Gamma distribution through independence of $U=X/Y$ and $V=X+Y$.

At the end of the 1990s, a new result of this kind, called the Matsumoto--Yor (MY) property was discovered; see, e.g., \cite[p. 43]{St05}. It states that if $X$ has a generalized inverse Gaussian (GIG) distribution and $Y$ is Gamma, the random variables  $U=1/(X+Y)$ and $V=1/X-1/(X+Y)$  are independent. It arose in studies \cite{MY01, MY03} of the conditional structure of some functionals of the geometric Brownian motion. See \cite{LW00} for a related characterization of the GIG and Gamma distributions through the independence of $X$ and $Y$ and of $U$ and $V$.

The MY property is also strongly rooted in classical multivariate analysis. Its matrix-variate version appears naturally in the conditional structure of Wishart matrices; see, e.g., \cite{MW06} as well as \cite{Bu98} and \cite{GH02}. A higher-dimensional version of the MY property and related characterization was studied in \cite{MW04}, where a Gamma-type multivariate distribution was obtained by connecting its density shape to a tree. Through a suitable transformation related to directed trees, a random vector having the latter distribution was mapped to independent components with GIG and Gamma distributions. This approach also led to a characterization of the product of GIG and Gamma distributions. In the special case of a chain with two vertices these results are equivalent to the characterization through the original MY property.

The MY property attracted a lot of attention in the last 15 years. In particular, \cite{KV12} tried to identify all possible functions $f$ and distributions of independent $X$ and $Y$ such that $f(X+Y)$ and $f(X)-f(X+Y)$ are also independent. The MY property corresponds to the case $f(x) = 1/x$. Another important case identified in that paper occurs when $f(x)=\ln(1+ 1/x)$. Thus if $X$ and $Y$ are independent with Kummer and Gamma distributions, then
$$
U=X+Y \quad \mbox{and} \quad V = \frac{1+{1}/(X+Y)}{1+{1}/{X}}
$$
are independent Kummer and Beta distributed random variables, respectively.

Characterizations by independence of $X$ and $Y$ and of $U$ and $V$ were obtained in \cite{KV11, KV12}. To derive their results, however, the authors needed to impose technical conditions of differentiability or local integrability of logarithms of strictly positive densities. Recently a regression characterization under natural integrability condition was given in \cite{We15} without any assumptions on the densities.

In the present paper we are interested in an independence property discovered recently in \cite{HV15}. It states that if $X$ and $Y$ are independent random variables with Kummer and Gamma distributions, then
$$
U=Y/(1+X) \quad \mbox{and} \quad V=X\,\frac{1+X+Y}{1+X}
$$
are also independent and have Kummer and Gamma distributions, respectively. In studying this property, we will exploit several of the ideas described above. In particular, we will give a characterization of the Gamma and Kummer distributions through the independence of the components in the pairs $(X,Y)$ and $(U,V)$. In the proof, inspired from \cite{We02a}, we will use the method of functional equations for densities assuming local integrability of their logarithms. This is reported in Section~\ref{sec:2}.

Next, we will introduce and study multivariate versions of the property described above. Our approach parallels the one adopted in \cite{MW04} for a multivariate version of the MY property. We will first define a $p$-variate tree-Kummer distribution from an undirected tree $T$ of size $p$. For each vertex $r$ of $T$, we will define the directed tree by choosing $r$ as its root. To each such directed tree, we will associate a transformation $\Phi_r:(0,\infty)^p\to(0,\infty)^p$ and show that if a random vector ${\bf S}$ has a tree-Kummer distribution, then $\Phi_r(\bf S)$ has independent components with Gamma and Kummer distributions. This analogue of Theorem~3.1 in \cite{MW04} is given in Section~\ref{sec:3}.

In Section~\ref{sec:4} we will derive a characterization of products of Gamma and Kummer distributions (and thus of the tree-Kummer distribution) assuming that for any leaf $r$ of the tree $T$, the components of $\Phi_r(\bf S)$ are independent. This result parallels Theorem~4.1 in \cite{MW04}. Finally, Section~\ref{sec:5} contains some concluding remarks.

\section{Kummer and gamma characterization}
\label{sec:2}

The Kummer distribution $\mathcal{K}(\alpha,\beta,\gamma)$ with parameters $\alpha,\gamma>0$, $\beta\in\R$ has density
$$
f(x)\propto \frac{x^{\alpha-1}}{(1+x)^{\alpha+\beta}}\,e^{-\gamma x}\,\mathbbm{1}_{(0,\infty)}(x).
$$
If $\beta>0$ it is a natural exponential family generated by the second kind Beta distribution. More information on the Kummer distribution, its  properties and applications can be found in \cite{AB97, AB14, CDONP, Dy53, KN95, KV12, Le09} and in the monograph \cite{BJK95}.

By the Gamma distribution $\mathcal{G}(\alpha,\gamma)$ with parameters $\alpha,\gamma>0$, we mean the distribution whose density is given by
$$
g(x)\propto x^{\alpha-1}e^{-\gamma x}\mathbbm{1}_{(0,\infty)}(x).
$$

Consider two independent random variables $X$ and $Y$ with respective distributions $X\sim \mathcal{K}(a,b-a,c)$ and $Y\sim \mathcal{G}(b,c)$, where $a,b,c>0$. Define a bijection $T:\; (0,\infty)^2\to (0,\; \infty)^2$ by
$$
T(x,y) = \left( \frac{y}{1+x}, x\left(1 + \frac{y}{1+x}\right)\right).
$$

Let $(U,V)=T(X,Y)$. It has been observed in \cite{HV15} that $U$ and $V$ are independent and $U\sim \mathcal{K}(b, a-b, c)$, $V\sim\mathcal{G}(a,c)$. Our objective in this section is to give a converse of this result, that is a characterization of the Kummer and the Gamma distribution through the independence property mentioned above. Unfortunately, as in the case of the characterization of the Kummer and Gamma distributions obtained by \cite{KV11, KV12}, we also need to impose some regularity conditions on densities.

\begin{theorem}
Let $X$ and $Y$ be two independent positive random variables with positive and continuously differentiable densities on $(0,\infty)$. Suppose that
$$ U = \frac{Y}{1+X},\; \; V = X\left(1 + \frac{Y}{1+X}\right),$$
are independent. Then there exist constants $a,b,c>0$, such that $ X\sim \mathcal{K}(a,b-a,c)$, $ Y\sim \mathcal{G}(b, c)$ or, equivalently, $U\sim \mathcal{K}(b, a-b, c)$ and $V\sim\mathcal{G}(a,c)$.
\label{tw1}
\end{theorem}

\begin{proof}
Note that $T$ is an involution. Given that $(U,V)=T(X,Y)$, we also have $(X,Y) = T(U,V)$. Given that the random vectors $(U, V)$ and $(X,Y)$ have independent components with continuous densities $p_U$, $p_V$, $p_X$ and $p_Y$ respectively, the independence property can be rewritten as
\bel{puu}
p_{U}(u) p_V(v) = |J(u,v)|p_X\left(\frac{v}{1+u}\right)p_Y\left\{u\left(1+\frac{v}{1+u}\right)\right\}
\ee
for all $x,y,u,v>0$, where $J$ is the Jacobian. Furthermore, given that
$$
J(u,v) = \frac{1}{1+u} \left(1 + \frac{v}{1+u}\right)
$$
and the densities are strictly positive it follows that Eq.~\eqref{puu} can alternatively be written as the functional equation
\bel{aubv}
A(u) + B(v)  =  C\left(\frac{v}{1+u}\right) + D\left\{u\left(1+\frac{v}{1+u}\right)\right\}
\ee
where
$$
A(x)  = \ln \,p_U(x) +\ln \,x,\qquad B(x)  = \ln \,p_V(x)+\ln \,x,
$$$$
C(x) = \ln \,p_X(x)+\ln \,x, \qquad D(x)  =  \ln \,p_Y(x)+\ln \,x.
$$
Differentiating both sides of \eqref{aubv} with respect to $u$ gives
$$
A'(u) =-\frac{v}{(1+u)^2} C'\left(\frac{v}{1+u}\right) + D'\left\{u\left(1+\frac{v}{1+u}\right)\right\} \left\{1+\frac{v}{(1+u)^2}\right\}.
$$
Now, we insert $(x,y) = T(u,v)$ in the above equation and get
\bel{eq_tu}
C'\left(x\right)x(1+x) = D'\left(y\right) \left\{1+x+y+x(1+x)\right\}-A'\left(\frac{y}{1+x}\right)(1+x+y).
\ee
Note that the right-hand side of Eq.~\eqref{eq_tu} converges to $\{D'(y)-A'(y)\}(1+y)$ as $x \to 0$. Hence, the left-hand side of Eq.~\eqref{eq_tu} also has a limit, say $-C_0$, when $x\to 0^+$ and $C_0$ does not depend on $y$ since there is no $y$ on the left-hand side of Eq.~\eqref{eq_tu}. Therefore,
$$
A'(y) =\frac{C_0}{1+y}+D'(y).
$$
We insert it back into Eq.~\eqref{eq_tu} to arrive at
\bel{eq_do_zbieg}
\frac{C_0}{x}+D'\left(\frac{y}{1+x}\right)\frac{1+x+y}{xy}\frac{y}{1+x} =-\frac{}{} C'(x)+ D'(y)\left\{ \frac{1+x+y}{x(1+x)}+1\right\}.
\ee

On the other hand with $u=y$ and $v=x(1+y)$, Eq.~\eqref{aubv} reads
\bel{ay}A(y)+B\{x(1+y)\}=C(x) +D\{y(1+x)\}.\ee
Differentiation with respect to $x$ yields
\bel{b'}B'\{x(1+y)\} (1+y)(1+x) -C'(x)(1+ x) = D'\{y(1+x)\}y(1+x) .\ee
Note that the left-hand side above has a finite limit when $y\to 0^+$.  Consequently, denoting $z=y(1+x)$ we conclude that $\lim_{z\to 0^+} D'(z) z=: b$ exists and is finite.
Therefore the left-hand side of Eq.~\eqref{eq_do_zbieg} is finite when $x\to\infty$. By comparing it with the right-hand side of \eqref{eq_do_zbieg} we conclude that $-\lim_{x\to\infty}C'(x)=:c$ is finite. Therefore, letting $x\to\infty$ in Eq.~\eqref{eq_do_zbieg}, we get
$$\frac{b}{y} = c + D'(y).$$
In view of the definition of $D$, we conclude that  $p_Y(y)\propto y^{b-1} e^{-cy}$ and $b,c>0$, i.e., $Y\sim\mathcal{G}(b,c)$.

Due to the fact that $T$ is an involution the functional equation \eqref{aubv} is symmetric in $(A,B)$ and $(C,D)$. Consequently, there exists a constant $a$ such that $B$ is of the form
$$
B'(x)  =\frac{a}{x} -d .
$$

In order to find $C$ let us note that Eq.~\eqref{b'} yields
$$
B'(x)(1+x) -C'(x)(1+x)=\lim_{y\to 0} D'(y(1+x))y(1+x)=b
$$
and thus
$$
C'(x) = \frac{a}{x}-d-\frac{b}{1+x}.
$$
Finally, we get
$$
C(x) = a\ln  x - d x -b\ln  (1+x) + c_0.
$$
Thus
$$
p_X(x)\propto \frac{x^{a-1}}{(1+x)^{b}}\,e^{-dx}\mathbbm{1}_{(0,\infty)}(x)
$$
and necessarily, $a>0$.

What is left to do is to determine the relationship between the parameters. Referring again to the definitions of $A$, $B$, $C$, $D$ and to the original Eq.~\eqref{puu}, we have
\begin{multline*}
p_U(u)\,  p_V(v) =\frac{1 + \frac{v}{1+u}}{1+u} p_X\left(\frac{v}{1+u}\right)p_Y\left\{u\left(1+\frac{v}{1+u}\right)\right\}
\\
\propto\frac{1+u+v}{(1+u)^2}\left(\frac{v}{1+u}\right)^{a-1}\left(\frac{1+u+v}{1+u}\right)^{-b} e^{-d\frac{v}{1+u}}\left(u\,\frac{1+u+v}{1+u}\right)^{b-1}e^{-cu\,\frac{1+u+v}{1+u}}$$
\\
=(1+u)^{-a}u^{b-1}e^{-cu}v^{a-1}e^{-cv}e^{-\frac{v}{1+u}(d-c)}.
\end{multline*}

Given that the function on the right-hand side has to be a product of a function of $u$ and a function of $v$, it follows that $d=c$. Thus, the result follows.
\end{proof}

To weaken the smoothness assumptions imposed on densities in Theorem \ref{tw1}, we will use local integrability instead of continuous differentiability, as proposed in \cite{We02a} for the MY-type functional equation, and then applied in \cite{KV11} in the characterization of the Kummer and Gamma distributions.

\begin{lemma}
It is sufficient to assume that logarithms of all densities are locally integrable in Theorem \ref{tw1}.
\end{lemma}

\begin{proof}
\setlength\arraycolsep{2pt}
Given that the other assumptions of Theorem \ref{tw1} are satisfied, we conclude that Eq.~\eqref{ay} holds for almost all $(x,y)\in (0,\infty)^2$. Given also that the functions $A$, $B$, $C$, $D$ are locally integrable we can take any  $0<x_0< x_1<\infty$ and integrate both sides of \eqref{ay} with respect to $x$ from $x_0$ to $x_1$. Then
\begin{eqnarray}
\int_{x_0}^{x_1}B\{x(1+y)\} dx -\int_{x_0}^{x_1}D\{y(1+x)\} \,dx & = & \int_{x_0}^{x_1} C(x) \,dx - A(y)(x_0-x_1).
\nonumber
\end{eqnarray}
We substitute $s = x(1+y)>0$ in the first integral  and  $t=y(1+x)>0$ in the second integral on the left-hand side above. Consequently, for almost all $y\in(0,\infty)$
\begin{eqnarray}
\int_{x_0(1+y)}^{x_1(1+y)}\frac{B(s)}{1+y} \, ds -\int_{(1+x_0)y}^{(1+x_1)y}\frac{D(t)}{y} \, dt &=& \int_{x_0}^{x_1} C(x) dx - A(y)(x_0-x_1).\label{cg_y}
\end{eqnarray}
The left-hand side of equation \eqref{cg_y} is continuous in $y\in (0,\infty)$. Hence, function $A$ can be extended to a continuous function $\tilde{A}$ on $(0,\infty)$.\\
Similarly, integrating \eqref{ay} with respect to $y$ from $y_0$ to $y_1$, $0<y_0<y_1<\infty $, we get
\begin{eqnarray}
\int_{(1+y_0)x}^{(1+y_1)x}\frac{B(s)}{x}\,ds -\int_{y_0(1+x)}^{y_1(1+x)}\frac{D(t)}{1+x}\,dt & = & (y_1-y_0) C(x) - \int_{y_0}^{y_1}A(y)dy\label{cg_x}
\end{eqnarray}
for almost all $x>0$. Therefore,  $C$ has also a continuous extension $\tilde{C}$ which is a function on $(0,\infty)$.

For any $x,y>0$ let $u= x(1+y)$ and $v = y(1+ x)$. Then
\begin{equation}
\label{dziala}
\left\{ \begin{array}{lll}
x &=& \frac{1}{2}\left(\sqrt{(1+v-u)^2 + 4u}-(1+v-u)\right)>0,\\
y & =& \frac{1}{2}\left(\sqrt{(1+u-v)^2 + 4v}-(1+u-v)\right)>0.
\end{array} \right.
\end{equation}
Plugging these values of $x$ and $y$ into Eq.~\eqref{ay} with $\tilde{A}=A$ and $\tilde{C}=C$, we get
\begin{multline}
B(u) - D(v) = \tilde{C}\left\{\frac{1}{2}\left(\sqrt{(1+v-u)^2 + 4u}-(1+v-u)\right)\right\} +
\\-\tilde{A}\left\{\frac{1}{2}\left(\sqrt{(1+u-v)^2 + 4v}-(1+u-v)\right)\right\}.
\label{cg_bd}
\end{multline}
The right-hand side of Eq.~\eqref{cg_bd} can be extended continuously for $u\in(0,\infty)$, which gives a continuous extension $\tilde{B}$ of $B$ on the left-hand side. Similarly, we can extend $D$ to $\tilde{D}$, which is continuous on $(0,\infty)$.

Consequently, we can assume that Eq.~\eqref{ay} is satisfied for all ${(x,y)\in(0,\infty)^2}$ and that the functions $A$, $B$, $C$ and $D$ appearing in \eqref{ay} are continuous. Now, repeating step by step the above reasoning, we see that $A$, $B$, $C$ and $D$ may be assumed continuously differentiable, i.e., all the assumptions of Theorem \ref{tw1} are satisfied.
\end{proof}

\section{Independence properties of tree-Kummer distribution}
\label{sec:3}

\subsection{A symmetrization of the independence property}
\label{sec:3.1}

The first step in approaching the tree-version of the MY property was to consider a symmetrized version of the classical MY property; see \cite{MW04}, pp. 686--687. Similarly here we start with a derivation of a symmetric version of the property of independence of Kummer and Gamma distributions from \cite{HV15}.

Consider independent random variables $X\sim \mathcal{K}(a,b-a,c)$ and $Y\sim \mathcal{G}(b,c)$ with parameters $a,b,c>0$, and $(U,V)=T(X,Y)$. Define a random vector $$(S_1,S_2)=\left(X,\frac{Y}{1+X}\right)=\left(\frac{V}{1+U},U\right).$$ Then $(X,Y)=(S_1,S_2(1+S_1))$ and $(V,U)=(S_1(1+S_2),S_2)$. For the density $p$ of $(S_1,S_2)$ we get
$$
p(s_1,s_2)=|J_{\psi^{-1}}(s_1,s_2)|p_X(x(s_1,s_2))\,p_Y(y(s_1,s_2))
$$
with $\psi(x,y)=\left(x, y/(1+x)\right)$ being a bijection on $(0,\infty)^2$. Therefore, $\psi^{-1}(s_1,s_2)=(x(s_1,s_2),\,y(s_1,s_2))=(s_1,\,(1+s_1)s_2)$ and thus
$$
p(s_1,s_2)\propto s_1^{a-1}s_2^{b-1}\,e^{-c(s_1+s_2+s_1s_2)}\mathbbm{1}_{(0,\infty)^2}(s_1,s_2).
$$

We will consider a more general distribution of this type by introducing three positive parameters $c_1$, $c_2$ and $c_{1,2}$  and writing
\bel{kt2}
p(s_1,s_2)\propto s_1^{a_1-1}s_2^{a_2-1}\,e^{-(c_1s_1+c_2s_2+c_{1,2}s_1s_2)}\mathbbm{1}_{(0,\infty)^2}(s_1,s_2).
\ee
This is the analogue of the density $f$ in \cite{MW04}, p. 687.

Let us also define two bijective mappings: $\Phi_r: (0,\infty)^2\rightarrow (0,\infty)^2$, $r\in\{1,2\}$, by
\begin{eqnarray*}
\Phi_1(s_1,s_2) &=& (s_{1,(1)},\,s_{2,(1)}) = \left(s_1\left(1 + \frac{c_{1,2}}{c_1}s_2 \right),\, s_2\right)=\left(s_1\left(1 + \frac{c_{1,2}}{c_1} s_{2,(1)}\right) ,\,s_2\right).
\\
\Phi_2(s_1,s_2) &=&(s_{1,(2)},s_{2,(2)}) = \left(s_1,\, s_2\left(1+\frac{c_{1,2}}{c_2}s_1\right)\right)= \left(s_1,\, s_2\left(1+\frac{c_{1,2}}{c_2}s_{1,(2)}\right)\right).
\end{eqnarray*}
They are analogs of the mappings $\psi_1$ and $\psi_2$ in \cite{MW04}, p. 686.

Assume that a random vector $(S_1,S_2)$ has the density $p$ as in \CG{Eq.}~\eqref{kt2}. Define $(X_{1,(1)}, X_{2,(1)})= \Phi_1(S_1,S_2)$ and $(X_{1,(2)}, X_{2,(2)})= \Phi_2(S_1,S_2)$. Then standard computations involving density transformation yield
\bel{x1}\left(\frac{c_{1,2}}{c_1}X_{1,(1)},\,\frac{c_{1,2}}{c_2}\,X_{2,(1)}\right)\sim \mathcal{G}(a_1,c')\otimes \mathcal{K}(a_2,a_1-a_2,c')\ee and
\bel{x2}\left(\frac{c_{1,2}}{c_2}X_{1,(2)},\,\frac{c_{1,2}}{c_2} X_{2,(2)}\right)\sim \mathcal{K}(a_1,a_2-a_1,c')\otimes\mathcal{G}(a_2,c'),\ee
where $c'= c_1c_2/{c_{1,2}}$ and $\mu\otimes\nu$ denotes a distribution which is a product of distributions $\mu$ and $\nu$.

Our aim in this section is to extend the above construction to any dimension. As in \cite{MW04} the language of undirected and directed trees will be very helpful in this context. Similarly to \cite{MW04}, Section 2, we need to provide background information regarding the tree language and certain facts about transformations.

\subsection{Trees and transformations}
\label{sec:3.2}

Let us recall that a graph $G=(V,E)$, where $V$ is the set of nodes and $E\subseteq \{\{u,v\}: u,v\in V, u\neq v\}$ is the set of edges, is called a \textit{tree}, if it is connected and acyclic.  The symbol $\deg_G(i)$ stands for the degree of vertex $i\in V$, i.e., $\deg _G(i) = |\{\ell:\; \{\ell,i\}\in E\}|$, where $|B|$ denotes a number of elements in a finite set $B$. A vertex of degree 1 is a \emph{leaf}.
We define a \emph{subtree} $S$ of a tree $T=(V, E)$ as the graph $S=(V(S), E(S))$, where $V(S)\subseteq V$ and $E(S) = \left\{ \{i,j\} \in E: i, j\in V(S)\right\}$. If $S$ is a subtree of $T$ we write $S\subseteq T$.

Let $T=(V,E)$ be a tree of size $p\geq 2$. For a fixed root $r\in V$, we direct $T$ from the root towards leaves and denote such a directed tree by $T_r$. Having the tree directed, we can say that node $i$ is a \textit{child} of vertex $j$ (or $j$ is a \textit{parent} of $i$) if and only if $\{i,j\}\in E$ and the tree is directed from $j$ to $i$ (note that every node, unless it is a leaf, has at least one child and every node but the root has exactly one parent). The set of all children of $i$ in $T_r$ will be denoted by $\mathfrak{C}_r(i)$  and the parent of vertex $i$ by  $\mathfrak{p}_r(i)$. We say that the undirected tree $T$ is the \emph{skeleton} of $T_r$.

\begin{remark}
Note that when drawing parallels between the result presented here and those in \cite{MW04}, one has to keep in mind that there, contrary to the tradition in the graph theory, the tree is directed in the opposite way: from leaves toward a chosen root. Therefore, the parents and children notions here and in \cite{MW04} have to be swapped.
\end{remark}

For any $r\in V$ (and hence $T_r$) and fixed set of parameters $c_{i,j}>0$, $i,j\in V$ ($c_{i,i}=:c_i$) we define \CG{a} transformation $\Phi^T_r: \mathbb{R}_+^p \rightarrow \mathbb{R}_+^p $ by
\bel{frt}\Phi_r^T (s_i, i\in V) = ( s_{i,(r)}, i\in V ),\ee where
\bel{gw}s_{i,(r)}=s_i\,\displaystyle{\prod_{j\in \mathfrak{C}_r(i)}\,\left(1+\frac{c_{i,j}}{c_i} s_{j,(r)}\right)},
\ee
where by convention an empty product is equal to $1$, i.e., if $V\ni i\ne r$ is a leaf then $s_{i,(r)}=s_i$. The definition  \eqref{gw} is inverse recursive with a starting point being any vertex with maximal distance from the root $r$. Usually we will write $\Phi_r$ instead of $\Phi_r^T$. Compare these mappings with the mappings $\psi_r$, $r\in V$, defined in (2.5) and (2.6) of \cite{MW04}.

It is easy to see that $\Phi_r$ is a bijection for any $r\in V$. The transformation $\Phi_r^{-1}: \; \mathbb{R}_+^p \rightarrow \mathbb{R}_+^p$ is given by ${\Phi_r^{-1}(y_j, j\in V)=(s_i, i\in V)}$, where
$$s_i = \frac{y_i}{\prod_{j\in \mathfrak{C}_r(i)}(1+\frac{c_{i,j}}{c_i}y_j)},\quad i\in V,$$
and again the product over the empty set is considered to be equal to $1$.

First we compute the Jacobian $J_r$ of  $\Phi_r^{-1}$. Note, that if we enumerate the nodes in such a way that each child has a number  greater than its parents, the Jacobi matrix is upper triangular. Thus, we only need to find the elements on its diagonal which are partial derivatives
$$
\frac{\partial s_i}{\partial y_i}=\left\{\prod_{j\in \mathfrak{C}_r(i)}\left(1+\frac{c_{i,j}}{c_i}y_j\right)\right\}^{-1}, \quad  i\in V.
$$
Considering that every node but the root has exactly one parent, it appears in exactly one partial derivative. Hence, we have
\bel{jakob}
J_r(y_i, i\in V)=\left\{\prod_{i\in V\setminus\{r\}}\,\left(1+\frac{c_{\mathfrak{p}_r(i),i}}{c_{\mathfrak{p}_r(i)}}y_i\right)\right\}^{-1}.
\ee

We will also need the identity given in the following proposition.
\begin{proposition}\label{tw_graf_og}
For any $r\in V$ if $(s_{i,(r)})_{i\in V}$ and $(s_i)_{i\in V}$ are related by \eqref{gw} then
\bel{summ}\sum_{m\in V}\, c_m s_{m,(r)} =\sum_{S\subseteq T}\prod_{i\in V(S)} \frac{s_i}{c_i^{\deg _S (i)-1}}\prod_{\{j,k\}\in E(S)} c_{j,k} .\ee
\end{proposition}
\begin{proof}
Fix a root $r\in V$ in the tree $T$. Let $\mathcal{S}^m_r$ denote the family of directed subtrees of $T_r$  with the root $m\in V$. By $\mathcal{S}(r,m)$ we denote the family of undirected trees which are skeletons (undirected trees) of (directed) trees from $\mathcal{S}^m_r$.

We will show, using mathematical induction, that for every $m\in V$
\bel{row_og}
\sum_{S\in \mathcal{S}(r,m)}\prod_{\{j,k\}\in E(S)} c_{j,k}\prod_{i \in V(S)} \frac{s_i}{c_i^{\deg _{S}(i)-1}} =c_m s_{m,(r)}.
\ee
Given that
$$
\sum_{S\subseteq T}\prod_{i\in V(S)} \frac{s_i}{c_i^{\deg _S (i)-1}}\prod_{\{j,k\}\in E(S)} c_{j,k} =\sum_{m\in V}\sum_{S\in\mathcal{S}(r,m)}\prod_{\{j,k\}\in E(S)} c_{j,k}\prod_{i \in V(S)} \frac{s_i}{c_i^{\deg _{S} (i)-1}},
$$
the identity \eqref{summ} follows from Eq.~\eqref{row_og}.

In order to prove \eqref{row_og} we will rely on the definition \eqref{gw}. The proof will be by induction with respect to the distance of the vertex $m$ from its farthest descendant.

We first consider the case $\mathfrak{C}_r(m)=\emptyset$. Then the left-hand side of \eqref{row_og} equals
$s_mc_m$ as the only element of $\mathcal{S}(r,m)$ is a trivial tree  $(\{m\},\emptyset)$. In this case  $s_{m,(r)}=s_m$ due to \eqref{gw}, and thus the identity \eqref{row_og} follows.

Now we need to consider the case $\mathfrak{C}_r(m)\ne \emptyset$ assuming that for all $n\in \mathfrak{C}_r(m)$ the identity \eqref{row_og} holds, i.e.,
\bel{indu}\sum_{S\in\mathcal{S}(r,n)}\prod_{\{j,k\}\in E(S)} c_{j,k}\prod_{i \in V(S)} \frac{s_i}{c_i^{\deg _{S} (i)-1}} =c_n s_{n,(r)}.\ee

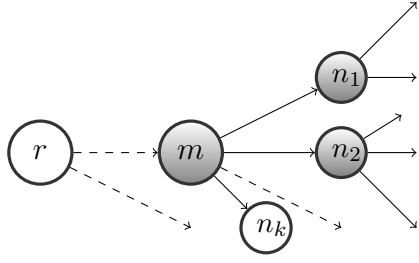
\begin{figure}[h]\begin{tikzpicture}
\tikzstyle{male} = [ font={\large}, shape=circle, minimum size=0.1cm,text=black,
very thick, draw=black!80,bottom color=black!50, top color=white, text width=0.25cm, align=left]
\tikzstyle{male1} = [ font={\large}, shape=circle, minimum size=0.1cm,text=black,
very thick, draw=black!80, top color=white, text width=0.25cm, align=left]
\tikzstyle{mynodestyle} = [ font={\Large\bfseries}, shape=circle, minimum size=0.2cm,text=black,
very thick, draw=black!80, top color=white, text width=0.5cm, align=center]
\tikzstyle{duzy} = [ font={\Large\bfseries}, shape=circle, minimum size=0.2cm,text=black,
very thick, draw=black!80, top color=white,bottom color=black!50, text width=0.5cm, align=center]
\node[ mynodestyle] (v1) at (-4,3) {$r$};
\node[ duzy] (v3) at (-2,3) {$m$};
\node[male] (v2) at (0,4)  {$n_1$};
\draw[dashed, ->] (v1) -- (v3);
\node [male] (v4) at (0,3) {$n_2$};
\draw[dashed, ->] (v1) -- (-2,2);
\draw[dashed, ->] (v3) -- (0,2);
\draw[->] (v2) -- (1,5);
\draw[ ->] (v2) -- (1,4);
\draw[->] (v4) -- (1,3);
\draw[->] (v4) -> (1,2);
\draw[->] (v4)-> (0.8,3.5);
\node[male1] (vk) at (-1,2) {$n_k$};
\draw[->] (v3) -- (v4);
\draw[->] (v3) -- (v2);
\draw[->] (v3) -- (vk);

\end{tikzpicture}\caption{Tree $T_r$ with one of possible $S\in \mathcal{S}(r,m)$ colored grey.}
\label{fig1}\end{figure}

Note that any element of $\mathcal{S}(r,m)$ is a union of subtrees from $\mathcal{S}(r,n)$, $n\in \mathfrak{C}_r(m)$, extended by attaching to such a union the vertex $m$  with suitably chosen edges (see Figure \ref{fig1} for illustration). For any $B\subset\mathfrak{C}_r(m)$ (ordered by numbers of vertices) we denote $\mathcal{S}(r,B)=\times_{w\in B}\,\mathcal{S}(r,w)$, a Cartesian product of families of subtrees. Denote also $S_B=(S_w)_{w\in B}\in\mathcal{S}(r,B)$ and note that it is a vector of subtrees. Using this notation we can write the family $\mathcal{S}(r,m)$ as follows
$$
\mathcal{S}(r,m) =
\left\{\left( V(S_B)\cup\{m\},\; E(S_B)\cup\bigcup_{w\in B}\,\{\{m,w\}\}\right),\; S_B\in\mathcal{S}(r,B),\;B\subset\mathfrak{C}_r(m) \right\},
$$
where
$$
V(S_B)=\bigcup_{w\in B}\, V(S_w) \quad \mbox{and} \quad E(S_B)=\bigcup_{w\in B}\,E(S_w).
$$
In particular, the element in the above family associated to $B=\emptyset\subset\mathfrak{C}_r(m)$ is the trivial tree $(\{m\},\emptyset)$. Then
\begin{multline}\label{aaa}
\sum_{S\in\mathcal{S}(r,m)}\prod_{\{j,k\}\in E(S)} c_{j,k}\prod_{i \in V(S)} \frac{s_i}{c_i^{\deg_S(i)-1}} \\
=c_ms_m +
\sum_{\emptyset\ne B\subset\mathfrak{C}_r(m)}\,\sum_{S_B\in\mathcal{S}(r,B)}\,\prod_{\{j,k\}\in\,E(S_B(m))}\,c_{j,k}\,\prod_{i\in V(S_B(m))}\, \frac{s_i}{c_i^{\deg_{S_B(m)}(i)-1}},
\end{multline}
where
$$
S_B(m)=\left(V(S_B(m)),\,E(S_B(m))\right)=\left(V(S_B)\cup\{m\},\; E(S_B)\cup\bigcup_{w\in B}\,\{\{m,w\}\}\right).
$$
Therefore,  $$\deg_{S_B(m)}(i)=\left\{\begin{array}{lll}|B| & \mbox{for} & i=m,\\ \deg_{S_w}(i) & \mbox{for} & i\in V(S_w)\setminus\{w\},\;w\in B,\\ \deg_{S_w}(w)+1 & \mbox{for} & i=w,\;w\in B.\end{array}\right.$$ Consequently, the double sum in \eqref{aaa} can be written as
$$
c_ms_m\sum_{\emptyset\ne B\subset\mathfrak{C}_r(m)}\,\frac{1}{c_m^{|B|}}\,\left(\prod_{w\in B}\,\frac{c_{m,w}}{c_w}\right)\,\sum_{S_B\in\mathcal{S}(r,B)}\,\prod_{w\in B}\,\left(\prod_{\{j,k\}\in E(S_w)}\,c_{j,k}\,\prod_{i\in V(S_w)}\, \frac{s_i}{c_i^{\deg_{S_w}(i)-1}}\right)
$$$$
=c_ms_m\sum_{\emptyset\ne B\subset\mathfrak{C}_r(m)}\,\frac{1}{c_m^{|B|}}\,\left(\prod_{w\in B}\,\frac{c_{m,w}}{c_w}\right)\,\prod_{w\in B}\,\left(\sum_{S_w\in\mathcal{S}(r,w)}\,\prod_{\{j,k\}\in\,E(S_w)}\,c_{j,k}\,\prod_{i\in V(S_w)}\,\frac{s_i}{c_i^{\deg_{S_w}(i)-1}}\right)
$$$$=c_ms_m\sum_{\emptyset\ne B\subset\mathfrak{C}_r(m)}\,\frac{1}{c_m^{|B|}}\,\prod_{w\in B}\,c_{m,w}\,s_{w,(r)},$$
where the last equality follows form the induction assumption \eqref{indu} applied to every sum under the product $\prod_{w\in B}$ in the previous expression. Summing up we have
$$
\sum_{S\in\mathcal{S}(r,m)}\prod_{\{j,k\}\in E(S)} c_{j,k}\prod_{i \in V(S)} \frac{s_i}{c_i^{d_{(r,m)}(i)}} =c_ms_m\left[1+\sum_{\emptyset\ne B\subset\mathfrak{C}_r(m)}\,\frac{1}{c_m^{|B|}}\,\prod_{w\in B}\,c_{m,w}\,s_{w,(r)}\right]
$$$$=c_ms_m\prod_{w\in\,\mathfrak{C}_r(m)}\left(1+\frac{c_{m,w}\,s_{w,(r)}}{c_m}\right)=c_ms_{m,(r)},
$$
where the last equation is a consequence of \eqref{gw}.
\end{proof}

Note that the right-hand side of \eqref{summ} does not depend on $r$, i.e., on the direction of the tree. Therefore the following result is an immediate consequence of Proposition~\ref{tw_graf_og}.
\begin{corollary}For any $r_1$, $r_2\in V$
$$\sum_{i\in V}c_i s_{i, (r_1)} = \sum_{i\in V}c_i s_{i, (r_2)}.$$
\end{corollary}

\subsection{Independence property of tree-Kummer distribution}\label{TKum}

Let $T=(V,E)$, as above,  be a tree of size $p$. We define the set $C_T$ as the set of all symmetric matrices ${\bf C}=[c_{i,j}]_{i,j\in V}$ such that $c_{i,j}>0$, when $\{i,j\}\in E$, $c_{i,i}=c_i>0$ and $c_{i,j}=0$ otherwise. We say that a random vector $\mathbf{X}$ has \emph{tree-Kummer} distribution, $\mathrm{TK}(\textbf{a}, \textbf{C})$, where  $\mathbf{a}=(a_i, i\in V)\in \mathbb{R}^p_+$, and $\textbf{C}\in C_T$   if its density is of the form
$$
f(\mathbf{x}) \propto\prod _{i\in V} x_i ^{a_i-1} \exp\left(- \sum_{S\subseteq T}\,C(S)\,\prod_{\{j,k\}\in E(S)} c_{j,k} \right)\mathbbm{1}_{(0;\infty)^p}(\mathbf{x}),
$$
where
$$
C(S)=\prod_{i\in V(S)} \frac{x_i}{c_i^{\deg_S(i)-1}}.
$$
As it will be  seen later, if $v\in V$ is a leaf then the law of $X_v$ is Kummer. On the other hand for a degenerate tree of size $1$, the $\mathrm{TK}$ distribution is just a Gamma distribution. This distribution is an analogue of the $W_G^c$ distribution introduced in \cite{MW04}, Section~3, which is also Gamma for the degenerate tree of size $1$, and if ${\bf X}$ is a random vector with $W_G^c$ distribution, then for any leaf $v\in V$ the random variable $X_v$ has a GIG distribution.

Note that in the case of the tree $T$ which is a 1--2, chain the above density agrees with \eqref{kt2}. Thus, as observed in Section~\ref{sec:3.1}, the transformations $\Phi_1$ and $\Phi_2$  applied to a bivariate random vector with such density produce random vectors with independent Kummer and Gamma components. We will extend this observation to a tree-Kummer law generated by any tree.  It will lead to  a multivariate version of the independence property from \cite{HV15}.

\begin{theorem}\label{tw_og}
Let $T=(V, E)$ be a tree, $|V|=p\geq 2$, and $\mathbf{S}=(S_i, i\in V)$ be a random vector
following the tree-Kummer $\mathrm{TK}(\mathbf{a}, \mathbf{C})$ distribution with $\textbf{a}\in \mathbb{R}^p_+$,  $\mathbf{C}\in C_T$.
Denote $\mathbf{X}_r = \Phi_r(\mathbf{K})$, where $\Phi_r$ is the transformation defined in \eqref{frt} and \eqref{gw} with parameters $c_{i,j}$ and $c_i$, where $i,j,r\in V$.

Then, for any $r\in V$ the components of the random vector ${\mathbf{X}_r=(X_{i,(r)},i\in V)}$ are independent. Moreover,
$${X_{r,(r)}\sim \mathcal{G}(a_r, c_r)}\quad\mbox{and} \quad \frac{c_{\mathfrak{p}_r(i),i}}{c_{\mathfrak{p}_r(i)}}\,X_{i,(r)}\sim \mathcal{K}\left(a_i,a_{\mathfrak{p}_r(i)}-a_i,\frac{c_{\mathfrak{p}_r(i)}\,c_i}{c_{\mathfrak{p}_r(i),i}}\right),\quad i\in V\setminus\{r\}.$$
\end{theorem}
This theorem may be viewed as an analogue of Theorem~3.1 in \cite{MW04}.

\begin{proof}
As for any $r\in V$ the map $\Phi_r$ is a diffeomorphism from $(0,\infty)^p$ onto $(0,\infty)^p$, the random vector $\mathbf{X}_r$ has a density $f_r$ of the form $$f_r(x_i, i\in V) = |J_r(x_i, i\in V)|f_T\left(\Phi^{-1}_r(x_i, i\in V)\right)\mathbbm{1}_{(0;\infty)^p}(\mathbf{x}).
$$
Proposition \ref{tw_graf_og} implies that for every vertex $r \in V$,
$$
\sum_{S\subseteq T}\prod_{i\in V(S)} \frac{x_i}{c_i^{\deg _S (i)-1}}\prod_{\{l,j\}\in E(S)} c_{l,j} =\sum_{i\in V} c_i \Phi _r^{(i)} (\mathbf{x}),
$$
where by $\Phi_r^{(i)}(\mathbf{x})$ we mean $x_{i,(r)}$ as defined in Eq.~\eqref{gw}, so ${\Phi_r(\cdot)=(\Phi_r^{(i)}(\cdot), i\in V)}$).

Thus, for $x_i>0$ and $i\in V$
\begin{eqnarray*}
f_T \{ \Phi_r^{-1}(x_i, i\in V)\} & \propto & \prod_{i\in V}
\left\{ \left(\Phi_r^{(-1)}\right)^{(i)}(\mathbf{x})\right\}^{a_i-1} \, \exp\left[ - \sum_{i\in V} c_i \Phi _r^{(i)} \{\Phi_r^{-1}(\mathbf x)\}\right]\\
&=& \prod_{i\in V} \left(\frac{x_i}{\prod_{j\in \mathfrak{C}_r(i)}(1+\frac{c_{i,j}}{c_i} x_j)}\right)^{a_i-1}\, \exp\left(-\sum_{i\in V} c_i x_i\right).\end{eqnarray*}

To compute the denominator of the above expression, note that for any set of numbers $(\alpha_{i,j},\,\{i,j\}\in E)$,
$$
\prod_{i\in V}\prod_{j\in \mathfrak{C}_r(i)} \alpha _{i,j} = \prod_{i\in V\setminus \{r\}} \alpha_{\mathfrak{p}_r(i),i}.
$$
Therefore,
$$
f_T \{ \Phi_r^{-1}(x_i, i\in V) \} \propto x_r^{a_r-1}e^{-c_r x_r} \prod_{i\in V\setminus\{r\}}  e^{ c_i x_i} x_i^{a_i -1}\left(1+\frac{c_{\mathfrak{p}_r(i),i}}{c_{\mathfrak{p}_r(i)}} x_i\right)^{-a_{\mathfrak{p}_r(i)}+1}.
$$
Calling also on the formula \eqref{jakob} for the Jacobian $J_r$, we see that the density of $\Phi_r(\mathbf{X})$ at $\mathbf{x}\in (0,\infty)^p$ is such that
$$
f_r(\mathbf{x}) \propto \,x_r^{a_r-1}e^{-c_r x_r}\mathbbm{1}_{(0;\infty)}(x_r)\prod_{i\in V\setminus\{r\}} x_i^{a_i-1} \left(1+\frac{c_{\mathfrak{p}_r(i), i}}{c_{\mathfrak{p}_r(i)}}x_i \right)^{-a_{\mathfrak{p}_r(i)}}e^{-c_i x_i}\mathbbm{1}_{(0;\infty)}(x_i).
$$
To conclude the proof, it suffices to note that the right-hand side above is a product of densities of distributions of $X_{i, (r)}$, $i\in V$, as stated in the theorem.
\end{proof}

\begin{remark}
Anticipating the characterization result of the next section, we remark that if in the above theorem we take ${\bf S}\sim\mathrm{TK}(\mathbf{a}, c\mathbf{C})$ for a positive number $c>0$, then with $\Phi_r$ defined by $\mathbf{C}$ as above we will get independent $X_{i,(r)}$, $i\in V$, with
$$
{X_{r,(r)}\sim \mathcal{G}(a_r, cc_r)}\quad\mbox{and} \quad \frac{c_{\mathfrak{p}_r(i),i}}{c_{\mathfrak{p}_r(i)}}\,X_{i,(r)}\sim \mathcal{K}\left(a_i,a_{\mathfrak{p}_r(i)}-a_i,c\frac{c_{\mathfrak{p}_r(i)}\,c_i}{c_{\mathfrak{p}_r(i),i}}\right),\quad i\in V\setminus\{r\}.
$$
\end{remark}

As in \cite{MW04}, we illustrate Theorem \ref{tw_og} with two examples.

\begin{example} Let $T$ be 1--2--3 chain and $c_1=c_2=c_3=c_{1,2}=c_{2,3}=c$. Let ${\bf S}=(S_1,S_2,S_3)$ have the respective tree-Kummer distribution and $\Phi_1, \Phi_2, \Phi_3$ are defined by Eq.~\eqref{frt} and \eqref{gw}. Then from Theorem~\ref{tw_og} we conclude that
\begin{multline*}
\Phi_1({\bf S})=\left(S_1(1+S_2(1+S_3)),\,S_2(1+S_3),\,S_3\right) \\
\sim\mathcal{G}(a_1,c)\otimes\mathcal{K}(a_2,a_1-a_2,c)\otimes\mathcal{K}(a_3,a_2-a_3,c),
\end{multline*}
and
$$
\Phi_2({\bf S})=\left(S_1,\,S_2(1+S_1)(1+S_3),\,S_3\right)\sim\mathcal{K}(a_1,a_2-a_1,c)\otimes\mathcal{G}(a_2,c)\otimes\mathcal{K}(a_3,a_2-a_3,c).
$$
\end{example}

\begin{example}
Let $T$ be a ``daisy'' on four vertices, $V=\{1,2,3,4\}$ with edges $E=\{$1--4, 2--4, 3--4$\}$ and $c_i=c_{i,j}=c$ for all $i\in V$. Assume that ${\bf S}=(S_1,S_2,S_3,S_4)$ has a respective tree-Kummer distribution and $\Phi_r$, $r\in V$, are defined by Eq.~\eqref{frt} and \eqref{gw}. Then from Theorem~\ref{tw_og} we conclude that
\begin{multline*}
\Phi_1({\bf S})=\left(S_1(1+S_4(1+S_2)(1+S_3)),\,S_2,\,S_3,\,S_4(1+S_2)(1+S_3)\right) \\ \sim\mathcal{G}(a_1,c)\otimes\mathcal{K}(a_2,a_4-a_2,c)\otimes
\mathcal{K}(a_3,a_4-a_3,c)\otimes\mathcal{K}(a_4,a_1-a_4,c)
\end{multline*}
and
\begin{multline*}
\Phi_4({\bf S})=\left(S_1,\,S_2,\,S_3,\,S_4(1+S_1)(1+S_2)(1+S_3)\right) \\ \sim\mathcal{K}(a_1,a_4-a_1,c)\otimes\mathcal{K}(a_2,a_4-a_2,c)\otimes
\mathcal{K}(a_3,a_4-a_3,c)\otimes\mathcal{G}(a_4,c).
\end{multline*}
\end{example}

\section{Multivariate characterization of Kummer and Gamma laws}
\label{sec:4}

Let us consider first a symmetrized and slightly generalized version of Theorem~\ref{tw1}, which can be viewed as a characterization of the simplest tree-Kummer distribution for the tree being a 1--2 chain.

\begin{theorem}\label{sym2} Let $(S_1,S_2)$ be a random vector having a positive and continuously differentiable density $f$ on $(0,\infty)^2$. For positive numbers $c_1$, $c_2$ and $c_{1,2}$ define the bijections $\Phi_1$ and $\Phi_2$ as in Section~\ref{sec:3.1}. Assume that random vectors $\mathbf{X}_{(1)} = (X_{1,(1)}, X_{2,(1)})=\Phi_1(S_1,S_2)$ and $\mathbf{X}_{(2)}=(X_{1,(2)}, X_{2,(2)})=\Phi_2(S_1,S_2)$ have independent components. Then there exist positive numbers $a_1$, $a_2$ and $c$ such that
$$
f(s_1,s_2)\propto s_1^{a_1-1}s_2^{a_2-1}\,e^{-c(c_1s_1+c_2s_2+c_{1,2}s_1s_2)}\,\mathbbm{1}_{(0;\infty)^2}(s_1,s_2).
$$
Equivalently,
$$
X_{1,(1)}\sim\mathcal{G}(a_2,cc_1),\qquad\,\frac{c_{1,2}}{c_1}\,X_{2,(1)}\sim\mathcal{K}\left(a_1,a_2-a_1,c\frac{c_1c_2}{c_{1,2}}\right)
$$
or
$$
\frac{c_{1,2}}{c_2}X_{2,(1)}\sim\mathcal{K}\left(a_2,a_1-a_2,c\frac{c_1c_2}{c_{1,2}}\right),\qquad\,X_{2,(2)}\sim\mathcal{G}(a_1,cc_2).
$$
\end{theorem}

\begin{proof}
Note that $X= c_{1,2}X_{2,(1)}/c_1$ and $Y=c_{1,2}X_{1,(1)}/c_2$ are independent. Define $U={Y}/(1+X)$ and $V=X(1+U)$. Then $U={c_{1,2}}X_{1,(2)}/c_1$ and $V={c_{1,2}}X_{2,(2)}/c_2$. Consequently, they are also independent. Now by Theorem~\ref{tw1} it follows that there exist positive numbers $a_1$, $a_2$ and $c'$ such that $X\sim \mathcal{K}(a_1,a_2-a_1,c')$ and $Y\sim\mathcal{G}(a_2,c')$; furthermore, $U\sim\mathcal{G}(a_1,c')$ and $V\sim \mathcal{K}(a_2,a_1-a_2,c')$. Then it is easily seen that the joint density $f$ of $(S_1,S_2)$ has the form
$$
f(s_1,s_2)\propto s_1^{a_1-1}s_2^{a_2-1}e^{-c'\frac{c_{1,2}}{c_1c_2}\left(c_1s_1+c_2s_2+c_{1,2}s_1s_2\right)}\,\mathbbm{1}_{(0;\infty)^2}(s_1,s_2),
$$
and the result follows upon taking $c=c'{c_{1,2}}/(c_1c_2)$.
\end{proof}

The above result provides a characterization of the bivariate tree-Kummer distribution of ${\bf S}=(S_1,S_2)$ for a tree being the 1--2 chain. Equivalently, this result can be interpreted as a characterization of univariate Kummer and Gamma distributions by considering the (properly scaled) components of $\Phi_1({\bf S})$ and $\Phi_2({\bf S})$.

The aim of this section is to extend the above characterization to any tree. The result we obtain is analogous to the characterization of GIG and Gamma distribution obtained in a tree setting in \cite{MW04}. The analogue of the tree-Kummer distribution used here is a tree-GIG distribution called the $W^c_G$ distribution in \cite{MW04}.

\begin{theorem}\label{tw_char}
Let $T=(V,E)$ be a tree of size $p$. Let $\mathbf{C}\in C_T$ and $\Phi_r$ be defined by \eqref{frt} and \eqref{gw}. Let $\mathbf{X}=(X_i, i\in V)$ be a $p$-dimensional random vector with positive density, which is continuously differentiable on its support $(0,\infty)^p$. Suppose that for every  $r\in V$, which is a leaf of $T$, the components of the random vector $\mathbf{X}_{(r)}=\Phi_r(\mathbf{X})= (X_{i,(r)},i\in V)$ are independent. Then there exist $\textbf{a}=(a_i, i\in V)\in (0,\infty )^p$ and $c>0$ such that $X_{r,(r)}\sim \mathcal{G}(a_r, cc_r)$ and
$$
\frac{c_{\mathfrak{p}_r(i),i}}{c_{\mathfrak{p}_r(i)}}\,X_{i,(r)}\sim \mathcal{K}\left(a_i,a_{\mathfrak{p}_r(i)}-a_i,c\frac{c_{\mathfrak{p}_r(i)}\,c_i}{c_{\mathfrak{p}_r(i),i}}\right)\; \mbox{for } i\in V\backslash \{r\}.
$$ Consequently, $\textbf{X}\sim \mathrm{TK}(\mathbf{a}, c \mathbf{C})$.
\end{theorem}

\begin{proof}
The proof will be by induction with respect to $p=|V|$. The case $p=2$ is given in Theorem~\ref{sym2}.  For fixed $p>2$, we assume the assertion of the theorem to be true for every tree of size up to $p-1$ and we consider an arbitrary tree of size $p$.

Let $n$ and $m$ be vertices from the set of leaves $L\subset V$ of tree $T$ (of course, $L$ consists of at least two elements). Given that $n$ and $m$ are leaves, each of them has precisely one neighbor, $n_1$ and $m_1$, respectively.

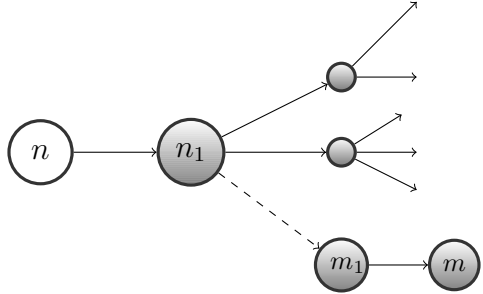
\begin{figure}[h]
\begin{tikzpicture}
\tikzstyle{male} = [  shape=circle, minimum size=0.05cm,text=black,
very thick, draw=black!80,  top color=white,bottom color=black!50,  text width=0.3cm, align=center]
\tikzstyle{male1} = [  shape=circle, minimum size=0.05cm,text=black,
very thick, draw=black!80, top color=white,bottom color=black!50,   text width=0.05cm, align=center]
\tikzstyle{mynodestyle} = [ font={\Large\bfseries}, shape=circle, minimum size=0.1cm,text=black,
very thick, draw=black!80,text width=0.5cm, align=center]
\tikzstyle{duzy} = [ font={\Large\bfseries}, shape=circle, minimum size=0.2cm,text=black,
very thick, draw=black!80, top color=white,bottom color=black!50, text width=0.5cm, align=center]
\node[ mynodestyle] (v1) at (-4,3) {$n$};
\node[ duzy] (v3) at (-2,3) {$n_1$};
\node[male1] (v2) at (0,4) {};
\draw[ ->] (v1) -- (v3);
\node [male1] at (0,3) (v4){};
\node (m1) [male] at (0,1.5) {$m_1$};
\node (m) [male] at (1.5,1.5){$m$};
\draw[->, dashed] (v3)--(m1);
\draw[->] (v2) -- (1,5);
\draw[ ->] (v2) -- (1,4);
\draw[->] (v4) -- (1,3);
\draw[->] (v4) -> (1,2.5);
\draw[->] (v4)-> (0.8,3.5);
\draw[->] (v3) -- (v4);
\draw[->] (v3) -- (v2);
\draw[->] (m1) -- (m);
\end{tikzpicture}\caption{$T_n$ with the subtree $\widetilde{T}_{n_1}$ colored grey.}\label{fig.2}
\end{figure}

In order to make use of the induction assumption we will remove one vertex, $n$, from $T$ together with the edge $\{n,n_1\}$ and denote such a reduced tree of size $p-1$ by $\widetilde{T}$. Obviously, $\widetilde{T}$ is a tree. Let us introduce a random vector $\mathbf{\widetilde X}= (\widetilde{X}_{j})_{j\in V\setminus\{n\}}$, where  $\widetilde X_j = X_j$ for $j\in V\setminus\{n,n_1\}$ and $ \widetilde X_{n_1} = X_{n_1} (1 + {c_{n_1 n}}{} X_n/c_{n_1})$. Note that $\mathbf{\widetilde X}$ takes on values in $(0,\infty)^{p-1}$.

Let $\widetilde{\Phi}_r$ , $r\in V\setminus \{n\}$, be a transformation defined by $\widetilde T$ according to \eqref{gw} with the same parameters $c_{i,j}$, $c_i$ but for $i,j\in V\setminus \{n\}$. We set
$$
\widetilde {\textbf{X}}_{ (r)}=( \widetilde X_{1, (r)},\ldots, \widetilde X_{p-1, (r)})= \widetilde{\Phi}_r(\widetilde{\textbf{X}}).
$$

Let $n_1$ be the root in $\widetilde{T}$. Clearly, $\widetilde T_{n_1}$ is a directed subtree of $T_n$; see Figure \ref{fig.2}. Therefore  $\mathfrak{C}_{n_1}^{\widetilde{T}}(n_1) = \mathfrak{C}_n^T(n_1)$ and $\widetilde X_{j, (n_1)}=X_{j, (n)} $ for $j\neq n_1$. Using these facts and \eqref{gw} we deduce that
\begin{multline*}
 \widetilde X_{n_1, (n_1)} =X_{n_1}\left(1+\frac{c_{n_1,n}}{c_{n_1}} X_n\right)\prod_{j\in \mathfrak{C}^{\widetilde T}_{n_1}(n_1)} \left(1+\frac{c_{n_1,j}}{c_{n_1}}\widetilde X_{j, (n_1)}\right) \\ =\left(1+\frac{c_{n_1 n}}{c_{n_1}} X_n\right)X_{n_1}\prod_{j\in \mathfrak{C}^T_n(n_1)} \left(1+\frac{c_{n_1,j}}{c_{n_1}}X_{j, (n_1)}\right)
=X_{n_1, (n)} \left(1+\frac{c_{n_1 n}}{c_{n_1}} X_n\right).
\end{multline*}
Consequently,
$$\widetilde X_{n_1, (n_1)} = \left(1+\frac{c_{n_1,n}}{c_{n_1}} \frac{X_{n, (n)}}{1+\frac{c_{n_1,n}}{c_n} X_{n_1, (n)}}\right)\,X_{n_1, (n)} .
$$
The latter equality, together with the independence of the components of $\textbf{X}_{(n)}$, entails the mutual independence of the components of $\widetilde {\mathbf{X}}_{(n_1)}$.

We also note that for any $\ell \in L\setminus \{n\}$, $X_{j, (\ell)} = \widetilde X_{j, (\ell)}$ for $j\in V\setminus\{n\}$, i.e., $\widetilde{X}_\ell$ is a subvector of ${\bf X}_\ell$ and thus
the components of $\widetilde{\mathbf{X}}_{(\ell)}$ are also independent.
The set of leaves of $\widetilde{T}$ is a subset of $L\cup \{n_1\}$, so eventually we get a $(p-1)$-dimensional random vector $\widetilde{\textbf{X}}$, for which the assumptions of Theorem~\ref{tw_char} are satisfied. Hence,  in particular, there exists a vector of positive components ${\bf a}^{(n)}=(a_i^{(n)},\,i\in V\setminus\{n\})$ and $c^{(n)}>0$ such that for $j\notin \{m,n\}$,
\bel{jeden}
\frac{c_{\mathfrak{p}_m(j),j}}{c_{\mathfrak{p}_m(j)}}\,\widetilde X_{j, (m)} = \frac{c_{\mathfrak{p}_m(j),j}}{c_{\mathfrak{p}_m(j)}}\,X_{j, (m)}\sim  \mathcal{K}\left(a_j^{(n)}, a_{\mathfrak{p}_m(j)}^{(n)} - a_j^{(n)}, c^{(n)}\frac{c_{\mathfrak{p}_m(j)}\,c_j}{c_{\mathfrak{p}_m(j),j}}\right),
\ee
(here we used the fact that $\mathfrak{p}^{\widetilde T}_m(j)=\mathfrak{p}^T_m(j)=:\mathfrak{p}_m(j)$ for $j\notin \{m,n\}$),
\bel{dwa}
\widetilde X_{m, (m)} = X_{m, (m)}\sim \mathcal{G}(a_m^{(n)},c^{(n)} c_m)
\ee
and for $j\notin \{n, n_1\}$,
\bel{trzy}
\frac{c_{\mathfrak{p}_n(j),j}}{c_{\mathfrak{p}_n(j)}}\,\widetilde X_{j, (n_1)} =\frac{c_{\mathfrak{p}_n(j),j}}{c_{\mathfrak{p}_n(j)}}\, X_{j, (n)}\sim  \mathcal{K}\left(a_j^{(n)}, a_{\mathfrak{p}_n(j)}^{(n)} - a_j^{(n)}, c^{(n)}\frac{c_{\mathfrak{p}_n(j)}\,c_j}{c_{\mathfrak{p}_n(j),j}}\right)
\ee
(here we used the fact that $\mathfrak{p}^{\widetilde T}_{n_1}(j)=\mathfrak{p}^T_n(j)=:\mathfrak{p}_n(j)$ for $j\notin \{n_1,n\}$ and that it does not matter if $n_1$ is  a leaf in $\widetilde{T}$).

Changing roles of nodes $n$ and $m$ (i.e., removing vertex $m$ instead of $n$ from tree $T$) and then proceeding as before for the random vector
$\widetilde{ \textbf{X}} =(\widetilde X_j)_{j\in V\setminus\{m\}}$, where $\widetilde X_j=X_j$ for $j\notin \{m, m_1\}$ and $\widetilde X_{m_1}=X_{m_1} (1+ {c_{m_1,m}}X_m/c_{m_1})$ for $j\notin \{ m,n\}$, we get
\bel{cztery}
\frac{c_{\mathfrak{p}_n(j),j}}{c_{\mathfrak{p}_n(j)}}\,\widetilde X_{j, (n)} =\frac{c_{\mathfrak{p}_n(j),j}}{c_{\mathfrak{p}_n(j)}}\, X_{j, (n)}\sim  \mathcal{K}\left(a_j^{(m)}, a_{\mathfrak{p}_n(j)}^{(m)} - a_j^{(m)}, c^{(m)}\frac{c_{\mathfrak{p}_n(j)}\,c_j}{c_{\mathfrak{p}_n(j),j}}\right),
\ee
\bel{piec}
\widetilde X_{n, (n)} = X_{n, (n)}\sim \mathcal{G}(a_n^{(m)}, c^{(m)}c_n),
\ee
and for $j\notin \{m, m_1\}$,
\bel{szesc}
\frac{c_{\mathfrak{p}_m(j),j}}{c_{\mathfrak{p}_m(j)}}\,\widetilde X_{j, (m_1)} = \frac{c_{\mathfrak{p}_m(j),j}}{c_{\mathfrak{p}_m(j)}}\,X_{j, (m)}\sim  \mathcal{K}\left(a_j^{(m)}, a_{\mathfrak{p}_m(j)}^{(m)} - a_j^{(m)}, c^{(m)}\frac{c_{\mathfrak{p}_m(j)}\,c_j}{c_{\mathfrak{p}_m(j),j}}\right).\ee

In the next step we are going to show that $a_j^{(n)}=a_j^{(m)}$ for every $j\in V$ and that $c^{(n)}=c^{(m)}$ for any leaves $m$, $n$.

To see that $c^{(n)}=c^{(m)}$ for any leaves $m$, $n$, it suffices to compare third parameters in \eqref{jeden} and \eqref{szesc}  --- or \eqref{trzy} and \eqref{cztery} --- since we chose above arbitrary leaves $m,n$. Furthermore, it follows from \eqref{trzy} and \eqref{cztery}  that $a_j^{(n)} = a_j^{(m)}$ for any $j\notin \{m, n, n_1\}$ and from \eqref{jeden} and \eqref{szesc} that $a_j^{(n)} = a_j^{(m)}$  for $j\notin  \{m, n, m_1\}$ for any choice of leaves $m$, $n$.

Note that in any tree except the "daisy" one can choose leaves $m$, $n$ in such a way that $n_1\ne m_1$. Then $a_j^{(n)}=a_j^{(m)}=:a_j$ for any $j\notin \{m,n\}$. For $j=n$, $m$ we just write $a_m:=a_m^{(n)}$ and $a_n:=a_n^{(m)}$.

If $n_1=m_1$ for any leaves $n,m$, we show the equality $a_{n_1}^{n}=a_{n_1}^m$ using the assumption of independence of components of $\mathbf{X}_{(r)}$ for $r\in L$. We have
\bel{eq_ind}
|J_{\Phi_m}(s_i, i\in V)| \prod_{i\in V}f_{i,(m)}(s_{i,(m)}) =|J_{\Phi_n}(s_i, i\in V)| \prod_{i\in V}f_{i,(n)}(s_{i,(n)})
\ee
where $f_{i,(\ell)}$ is the density of $X_{i,(\ell)}$, $i\in V$, $\ell = m,n$.

Now we compare the powers of $s_{n_1}$ on the left and the right-hand side of \eqref{eq_ind}. Note that $s_{n_1}$ appears only in
$$
f_{n_1, (\ell)}\left(s_{n_1}\prod_{j\in\C_\ell (n_1)}(1+\frac{c_{n_1,j}}{c_{n_1}}s_{j,(\ell)})\right),
$$
for all $\ell \in \{m,n\}$. In particular, see \eqref{cztery} and \eqref{jeden}, it is raised to the power $a_{n_1}^{(m)}$ in $f_{n_1, (m)}$ and to the power $a_{n_1}^{(n)}$ in $f_{n_1, (n)}$. Consequently, \eqref{eq_ind} yields $a_{n_1}^{(m)}=a_{n_1}^{(n)}=:a_{n_1}$. Now if there are at least three leaves we conclude that $a_j^{(m)}=a_j$ for any $j$ and for any leaf $m$.

In the case of the chain 1--2--3 no comparisons between the distributions from \eqref{jeden}--\eqref{szesc} is possible. Therefore to identify all the parameters we plug the densities from \eqref{jeden}--\eqref{szesc} directly into the identity \eqref{eq_ind} with $m=1$ and $n=3$. In this fashion, we can see that all the unknown parameters in \eqref{jeden}--\eqref{szesc} are derived from a single collection $(a_i)_{i\in V}$ and a number $c$. Therefore, due to Theorem \ref{tw_og}, we can write that, for any $r\in V$,
$$
X_{r,(r)}\sim G (a_r, cc_r),\;X_{i,(r)}\sim K\left(a_i,a_{\mathfrak{p}_r(i)}-a_i,cc_i, \frac{c_{\mathfrak{p}_r(i)}, i}{c_{\mathfrak{p}_r(i)}}\right) \textit{ for } i\in V\setminus \{r\}
$$ where $c>0$, $a_i>0$, $i\in V$.
\end{proof}

Again, as in \cite{MW04}, we give two examples to illustrate Theorem \ref{tw_char}.

\begin{example}
Consider chain 1--2--3 and related maps $\Phi_1$, $\Phi_3$ with all $c_i$'s and $c_{i,j}$'s equal to $1$. Let ${\bf S}$ be a random vector valued in $(0,\infty)^3$. Assume that the components of
$$\Phi_1({\bf S})=\left(S_1(1+S_2(1+S_3)),\,S_2(1+S_3),\,S_3\right)$$
are independent and that the components of
$$\Phi_3({\bf S})=\left(S_1,\,S_2(1+S_1),\,S_3(1+S_2(1+S_1)\right)$$
are also independent. Then from Theorem~\ref{tw_char}, the density of ${\bf S}$ is of the form
$$
f({\bf s})\propto s_1^{a_1-1}s_2^{a_2-1}s_3^{a_3-1}\,e^{-c(s_1+s_2+s_3+s_1s_2+s_2s_3+s_1s_2s_3)}\mathbbm{1}_{(0;\infty)^3}(\mathbf{s}).
$$
\end{example}

\begin{example}
Consider a ``three petal daisy,'' that is the tree with vertices $V=\{1,2,3,4\}$ and edges $E=\{$1--4, 2--4, 3--4$\}$ and related maps $\Phi_1$, $\Phi_2$, $\Phi_3$ with all $c_i$'s and $c_{i,j}$'s equal to 1. Let ${\bf S}$ be a random vector valued in $(0,\infty)^4$ with sufficiently smooth density. Assume that the components of
$$
\Phi_1({\bf S})=\left(S_1(1+S_4(1+S_2)(1+S_3)),\,S_2,\,S_3,\,S_4(1+S_2)(1+S_3)\right)
$$
are independent, the components of
$$
\Phi_2({\bf S})=\left(S_1,\,S_2(1+S_4(1+S_1)(1+S_3)),\,S_3,\,S_4(1+S_1)(1+S_3)\right)
$$
are independent, and the components of
$$
\Phi_3({\bf S})=\left(S_1,\,S_2,\,S_3(1+S_4(1+S_1)(1+S_2),\,S_4(1+S_1)(1+S_2)\right)
$$
are also independent. Then, from Theorem~\ref{tw_char}, the density of ${\bf S}$ is of the form
\begin{multline*}
f({\bf s})\propto s_1^{a_1-1}s_2^{a_2-1}s_3^{a_3-1}s_4^{a_4-1} \\
\times \,e^{-c(s_1+s_2+s_3+s_4+s_1s_4+s_2s_4+s_3s_4+ s_1s_2s_4+s_1s_3s_4+s_2s_3s_4+s_1s_2s_3s_4)}\mathbbm{1}_{(0;\infty)^4}(\mathbf{s}).
\end{multline*}
\end{example}

\section{Concluding remarks}
\label{sec:5}

In closing, we would like to remark that recently the approach of \cite{MW04} to the MY property on trees was extended to the matrix-variate case in \cite{Bo15}. To a large extent, this work was possible due to existing characterizations of the Wishart (the matrix analogue of the Gamma distribution) and the matrix GIG distributions based on the matrix version of MY property; see, e.g., \cite{Ko15, LW00, MW06, We02b}. In the case of Kummer and Gamma distributions, a matrix version of the independence property from \cite{KV12} was obtained in \cite{Kou12} and related characterization remains an open problem. A matrix version of the property and an extension of the characterization result from Section~\ref{sec:2} are currently under study.

It would also be interesting to know whether the property from \cite{HV15} can be embedded in the context of stochastic processes as was the case for the original MY property. As mentioned before, this was done in \cite{MY01} and \cite{MY03}. Similarly, the MY property on trees was related to the hitting times of the Brownian motion and  conditional structures of the geometric Brownian motion in \cite{WW07} and \cite{MWW09}. Therefore it is also of some interest to determine whether the tree version of the property from \cite{HV15} discussed here has also connections to stochastic processes.

\vspace{3mm}
{\bf Acknowledgement.} The authors are greatly indebted to Christian Genest and  Pawe{\l} Hitczenko for their considerable help regarding presentation of the material in this paper.

\end{document}